\documentclass[10pt]{amsart} 
\usepackage{amssymb}
\usepackage{graphicx}
\newtheorem{defi}{\bf D\scriptsize EFINITION \normalsize}
\newtheorem{theorem}{\bf T\scriptsize HEOREM \normalsize}
\newtheorem{lm}{\bf L\scriptsize EMMA \normalsize}
\newtheorem{dk}{\bf C\scriptsize OROLLARY \normalsize}
\newtheorem{rem}{\bf R\scriptsize EMARK \normalsize}
\newtheorem{exa}{\bf E\scriptsize XAMPLE \normalsize}
\newtheorem{pro}{\bf P\scriptsize ROBLEM \normalsize}
\newtheorem{prop}{\bf P\scriptsize ROPOSITION \normalsize}
\newtheorem{no}{\bf N\scriptsize OTE \normalsize}

\newenvironment{remark}{\begin{rem}\rm}{\end{rem}}
\def\kopr{\hfill\raisebox{3pt}{\framebox{$\star$}}}

\newenvironment{lemma}{\begin{lm}\it}{\end{lm}}

\usepackage{amssymb}
\begin{document}

\title{Genericity of distributional chaos in non-autonomous systems}

\author{Francisco Balibrea}
\address{Department of Mathematics, University of Murcia, Espinardo, 30100 Murcia, Spain}
\author{Lenka Ruck\'a$^1$}
\address{Mathematical Institute, Silesian University in Opava, Na Rybnicku 1, 74601 Opava, Czechia}
%\author{Francisco Balibrea$^1$ $~$ and $~$ Lenka Ruck\'a}%^1$}

\footnote{corresponding author}
{\hskip 115mm  June 11, 2024 
\vskip 5mm}

%\address{Department of Mathematics, University of Murcia, Espinardo, 30100 Murcia, Spain,\newline
%Mathematical Institute, Silesian University in Opava, Na Rybnicku 1, 74601 Opava, Czechia}

\email{balibrea@um.es}
\email{lenka.rucka@math.slu.cz}

\maketitle

\pagestyle{myheadings}
\markboth{{\sc F. Balibrea, L. Ruck\'a}}
{{\sc Genericity of distributional chaos in non-autonomous systems}}

%\begin{center} {\it \Large To the memory of Jaroslav Smítal.} \end{center}

\begin{abstract} 
In this paper we solve two open problems concerning distributional chaos in non-autonomous discrete dynamical systems stated in \cite{bss} and \cite{marta}. In the first problem it is wondered if the limit function of pointwise convergent non-autonomous system with positive topological entropy is DC2. We show that the answer to this problem depends on the given metric and can be both, positive or negative. In  the second open problem it is wondered if to be DC1 is a generic property of pointwise convergent non-autonomous systems. We prove that the answer is negative for convergent systems on the Cantor set. Concerning interval systems, we show that DC1 chaotic systems form dense, but not open (nor closed) set in the space of non-autonomous convergent systems on the interval, independently of the metric we use. 

\end{abstract}

\bigskip
\bigskip
\section{Introduction}

In modeling of real phenomena by dynamical systems, we can obtain models containing parameters described by whole families of maps, depending on time. Such models are generally challenging to handle, however after discretizing the time in the system, we deal with pairs $(X, f_{1,\infty})$ where $X$ is a space, $f_{1,\infty} = \{f_{n}\}_{n=1}^{\infty}$ and all $f_n$ are continuous maps acting on $X$. We obtain what is referred to in the literature as a non-autonomous discrete dynamical system. Such systems often have useful applications. One interesting example is a non-autonomous system $([-1,1], (T_{u_n})_{n=1}^{\infty})$ from \cite{silva}, composed of tent maps $T(x)$, cut by a constant function $u_{n} \in [-1,1]$. Such a system has been applied to problems in cardiology and telecommunications \cite{gar}, marketing \cite{he}, or population development \cite{hil}.

\medskip
One line of research deals with understanding the dynamics of non-autonomous systems and tries to extend results known for dynamical systems to non-autonomous case. This was begun by Kolyada and Snoha in \cite{kolyadasnoha} with topological entropy of two dimensional triangular discrete systems and afterwards this topic got widely studied, see for example \cite{paco}, \cite{balopr}, \cite{can} and the references therein. Since general non-autonomous system is difficult to work with, it is often required to be convergent. Most of the known results are for uniformly convergent systems on the interval. For example it is known that positive topological entropy of non-autonomous system is equivalent to positive topological entropy of its limit map (\cite{kolyadasnoha, bss}). For distributional chaos and Li-Yorke chaos it is known that chaoticity of the limit function implies chaoticity of the system \cite{marta}, but not the other way around \cite{can, dvor}. 
\medskip

On the other hand, so far only few relations are known for strictly convergent systems, see e.g. \cite{balopr}. In this paper we present results concerning distributional chaos and topological entropy for such non-autonomous systems, answering two open problems published in \cite{bss} and \cite{marta}. Our main result states that distributional chaos is generic (typical) property for convergent non-autonomous systems on the interval, while generic convergent system on the Cantor set is not distributionally chaotic. The analogous results concerning infinite topological entropy was published in \cite{bss}.
\medskip

\section{Preliminaries}
\label{prel}
\medskip

An {\it autonomous} discrete dynamical system ({\it ads}) is a pair $(X, f)$ of a compact metric space $X$ and con\-ti\-nu\-ous map $f:X \to X$ acting on it. Space $X$ is equipped with metric $d$. The trajectory of point $x \in X$ is a sequence $x, f(x), f^2(x), ...$, where $f^n(x) = f ^{n-1} \circ f(x)$ ($f^0$ is the identity map), denotes the $n$-th iterate of $f$. By $C(X)$ we denote space of all continuous maps from $X$ into itself, by $C_s(X)$ the set of all continuous and surjective maps. We use notation $I$ for real  compact unit interval, $M$ for compact manifold and $X$ for general compact space. \\

This paper deals with so called {\it non-autonomous} discrete dynamical systems ({\it nads} for short). It consists of a space $X$ (again, equipped with metric $d$) and a system of functions $f_{1, \infty}=\{f_i\}_{i=1}^\infty$, acting on it. We need all maps $f_i$ to be continuous and surjective. In the case of non-autonomous system, the trajectory of point $x \in X$ is the sequence $x, f_1(x), f_2\circ f_1(x), f_3 \circ f_2 \circ f_1(x), ...$. The $n$-th iterate $f_{1,\infty}^n(x)$ in this case is point $f_{n} \circ f_{n-1} \circ ... \circ f_2 \circ f_1(x)$. If system $f_{1, \infty}$ is convergent, we denote its limit map by $f$. The reason to suppose surjective maps is because if one of them, say $f_i$, was non-surjective, then the map $f_{i+1}$ could not be defined in all $X$. The assumption of surjectivity is sometimes extended to the possible limit function, although it is not necessary. \\

As well as in the case of {\it ads}, we can study dynamical behavior of {\it nads}. In this paper we solve problems focusing on distributional chaos and topological entropy. For definition of distributional chaos for autonomous systems we refer to \cite{3versions}, topological entropy for autonomous systems was studied and described in many papers and books, see e.g. \cite{alm}. Corresponding definitions for the setting of non-autonomous systems are very similar. We use definition of topological entropy from \cite{kolyadasnoha, balopr} and definition of distributional chaos from \cite{dvor, marta}. For wider background we refer to those papers. \\

Given positive integer $n$ and pair of points $x,y$, we denote 
\begin{equation}
\rho_n(x,y)=\max_{i=0,...,n-1} d(f_{1, \infty}^i(x), f_{1,\infty}^i(y)).
\label{pte}
\end{equation} 
A set $E \in X$ is called $(n,\varepsilon)-$separated, if $\rho_n(x,y)>\varepsilon$ for all pairs $x \neq y \in E$. Let $s_n(f_{1, \infty}, \varepsilon)$ be the maximal cardinality of $(n, \varepsilon)-$separated set. Finally the {\it topological entropy} of {\it nads} $f_{1, \infty}$ is defined as
\begin{equation}
h(f_{1, \infty})=\lim_{\varepsilon \to 0} \limsup_{n \to \infty} \frac1n \log (s_n(f_{1, \infty}, \varepsilon)).
\end{equation}
The property of positive topological entropy of a system is often referred to as PTE.\\

Let $\varepsilon>0$. For points $x,y \in X$ we define distributional functions $\psi_{xy}, \psi_{xy}^*: (0, \infty) \to [0,1]$ as follows:
\begin{equation}
\psi_{xy}(t):=\liminf_{n \to \infty} \frac1n \#\{0 \leq j <n; d(f_{1, \infty}^j(x), f_{1,\infty}^j(y))<t\} ~\hspace{1cm}~{\rm and}
\end{equation}
\begin{equation}
\psi^*_{xy}(t):=\limsup_{n \to \infty} \frac1n \#\{0 \leq j <n; d(f_{1, \infty}^j(x), f_{1, \infty}^j(y))<t\}.
\end{equation}
We call $\psi_{xy}$ the lower distributional function of $f_{1, \infty}$, $\psi_{xy}^*$ the upper distributional function of $f_{1, \infty}$ and as in the autonomous case, these are non-decreasing with $\psi_{xy}(t) \leq \psi_{xy}^*(t)$ for all $t$. 
\medskip

We say that pair of points $x,y$ form a DC1-scrambled pair of $f_{1, \infty}$, if 
\begin{equation}
\psi^*_{xy}(t) \equiv 1 \hspace{1cm} {\rm and} \hspace{1cm}  \psi_{xy}(t)=0 ~~~~ {~~~~\rm for~ some~~~~}~~~~ t>0.
\label{dc}
\end{equation}
The {\it nads} $f_{1, \infty}$ is called {\it distributionally chaotic of type 1} (or simply DC1), if there exists an uncountable subset $S \subseteq X$, where all pairs of points are DC1-scrambled pairs. \\

If distributional functions only satisfy 
\begin{equation}
\psi^*_{xy}(t) \equiv 1 \hspace{1cm} {\rm and} \hspace{1cm}  \psi_{xy}(t) <\psi^*_{xy}(t)  ~~~~ {~~~~\rm for~all~~~~}~~~~ t>0, 
\end{equation}
system $f_{1,\infty}$ is {\it distributionally chaotic of type 2} (DC2). There is also chaos DC3, which is even weaker than DC2, but we do not need the exact definition. For more information see \cite{3versions}. \\

The non-autonomous system $f_{1, \infty}=\{f_n\}_{n=1}^\infty$ is called {\it equicontinuous}, if
\begin{equation}
 \forall \varepsilon~~ \exists \delta>0 {\rm ~~such~~that~~} \forall n \in \mathbb{N} {\rm ~~ and ~~} \forall x,y \in X, ~~ \rho(x,y)<\delta \implies \rho(f_n(x), f_n(y))<\varepsilon. 
\label{equicontinuity}
\end{equation}
\smallskip

We also recall notions of convergence. It is said that {\it nads} $f_{1, \infty}$ converges to $f$ in metric $\rho$, if
\begin{equation}
\forall \varepsilon ~~\exists N {\rm ~~such~~ that~~} \forall n>N, \rho(f_{n}, f)<\varepsilon.
\label{1}
\end{equation}
System $\{f_{1, \infty}\}$ converges pointwise (with respect to metric $d$ on space $X$), if
\begin{equation}
\forall x ~~~ \forall \varepsilon ~~\exists N {\rm ~~such~~ that~~} \forall n>N, d(f_{n}(x), f(x))<\varepsilon.
\label{2}
\end{equation}
And finally it converges uniformly, if
\begin{equation}
\forall \varepsilon ~~\exists N {\rm ~~such~~ that~~} \forall n>N, d(f_{n}(x), f(x))<\varepsilon \hspace{.2cm} \forall x \in X.
\label{3}
\end{equation}
\smallskip

At last, we say that property $P$ is {\it generic} (or typical) for elements of some compact metric space, if the set of elements possessing this property is residual (complement of countable union of nowhere dense sets). In the setting of compact metric spaces, genericity is equivalent to being dense and open (see \cite{bss}).

\bigskip

\section{Open problem 1}
\label{oppr1}

The problem was stated in \cite{marta} and \cite{bss}. It was inspired by the Downarowicz's result, saying that if $X$ is compact metric space and $f \in C(X)$, if $h(f)>0$, then the system $(X, f)$ is DC2-chaotic (\cite{down}).\\

{\bf Problem:}  Assume $(X, f_{1,\infty})$ has positive topological entropy and $f_{1, \infty}$ converges pointwise to map $f$ in $C_s(X)$. Is it DC2? The conjecture in \cite{marta} is, that for {\it nads} on interval, the limit map must have a DC2 pair. \\

The answer is known to be positive for uniformly convergent systems $f_{1, \infty}$. It was proved in \cite{kolyadasnoha} that for uniform convergence, $h(f_{1, \infty})>0$ implies $h(f)>0$. And by Downarowicz's result in \cite{down}, map $f$ is DC2 chaotic. That is why authors of \cite{marta}, \cite{bss} situated the problem in the case of pointwise convergence. 
\medskip

The answer to the problem depends on the chosen metric. In both papers (\cite{marta}, \cite{bss}) is considered space of all continuous and surjective maps $C_s(X)$, together with supremum metric $\rho_{{\rm sup}}$, such that 
\begin{equation}
\rho_{{\rm sup}}(f,g) = \sup_{x \in X} d(f(x), g(x)), 
\label{rhosup}
\end{equation}
where $f, g \in C_s(X)$ and $d$ is a metric on $X$. What authors of \cite{marta}, \cite{bss} overlooked, is that with such metric, there is no difference between pointwise and uniform convergence. All sequences convergent in $\rho_{{\rm sup}}$ converges uniformly. To see this, check the definitions in the end of section \ref{prel}. Definition (\ref{1}) together with (\ref{rhosup}) gives
\begin{equation}
\forall \varepsilon ~~\exists N {\rm ~~such~~ that~~} \forall n>N, \sup_{x \in X} d(f_{n}(x), f(x))<\varepsilon.
\end{equation}
Hence $d(f_{n}(x), f(x))<\varepsilon$ in all points $x \in X$, which implies (\ref{3}), the uniform convergence of $f_n$.\\

Consequently the answer to the problem when supposing ($C_s(X), \rho_{{\rm sup}}$) is positive. The choice of supremum metric is obvious, it ensures continuity and surjectivity of the limit map, while pointwise convergence itself is not enough for that. If, anyway, we equip the space with different metric, the answer to the open problem might be negative. Suppose metric $\rho$ on $C_s(X)$ which distinguishes between uniform and pointwise convergence in the sense that there exist $\rho$-convergent sequences of maps, which are converging pointwise with respect to metric $d$ on $X$ (i.e. satisfy (\ref{2})), but are not uniformly convergent (do not satisfy (\ref{3})). The negative answer, in such case, follows by the result in \cite{balopr}. Authors in this paper proved, that for arbitrary continuous interval map $f$ there is a {\it nads} $(I, f_{1, \infty})$ with infinite topological entropy converging (pointwise) to $f$. So to get a negative answer to our problem it is enough to consider map $f$ with zero topological entropy, hence not DC2, and use construction from \cite{balopr} to create non-autonomous system with infinite topological entropy, converging to $f$. 
\medskip

In \cite{balopr} is constructed a system $f_{1, \infty}$ having infinite entropy. Now we use a simplified version of \cite{balopr} obtaining a system having PTE, shown in Figure \ref{picture1}. 

\begin{figure}[h]
\includegraphics[scale=0.8]{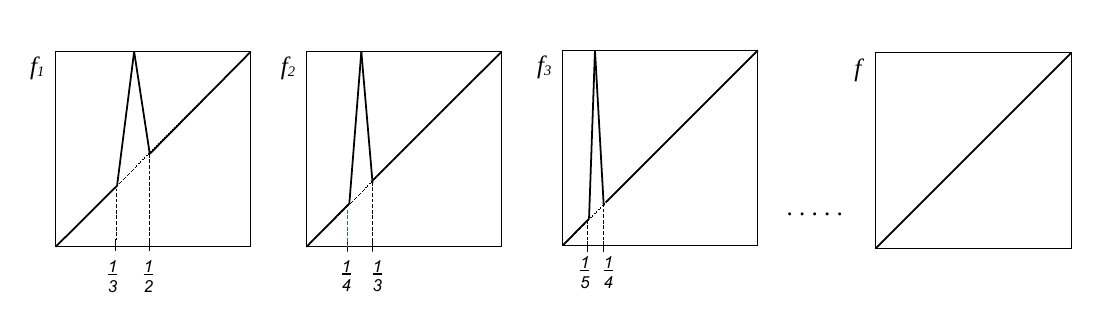} 
\caption{Non-autonomous system on the interval with PTE converging pointwise to the Identity function.}
\label{picture1}
\end{figure}

For completeness of the previous scheme, in \cite{rouyer} it was proved that if $S \subset X$ is generic with respect to a property, then $S$ is perfect and totally disconnected. Therefore $S$ is a Cantor set.
\medskip

Note, that now we suppose a metric which distinguishes between pointwise and uniform convergence. The pointwise limit of the system is the Identity map, which is not DC2. On the other hand the topological entropy of the system $f_{1, \infty}$ is positive. To see this, it is enough to imagine first few iterations of $f_{1, \infty}$ (in Figure \ref{2nd} there is graph of $f_2 \circ f_1)$. Second iteration ($f_2 \circ f_1$) has 4 peaks, for third one it is $3 \cdot 4+1=13$ and generally $n-$th iteration has 3 times plus one peaks more than the previous one. Since topological entropy is global chaotic behavior, it does not matter that the peaks are moving to the left all the time and $(n, \varepsilon)$-separated sets in (\ref{pte}) consists from different points every iteration. The important thing is that its cardinality is growing sufficiently fast.

\begin{figure}[h]
\includegraphics[scale=0.8]{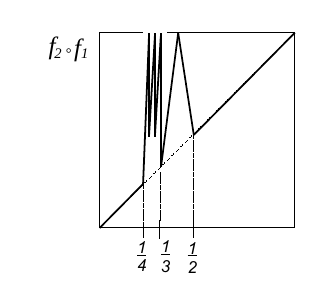} 
\caption{Second iterate of system $f_{1, \infty}$ from Figure \ref{picture1}}
\label{2nd}
\end{figure}

Sequence from Figure \ref{picture1} can be also used as an example of a {\it nads} on the interval which has PTE, but is not DC1 (or DC2, DC3), which is not possible in autonomous case, or for uniformly convergent {\it nads}. To see that $f_{1, \infty}$ is not distributionally chaotic (of any type), notice that all points of the system are eventually fixed. Indeed, points from interval $[\frac12, 1]$ are fixed points of $f_{1, \infty}$. Points from interval $[\frac13, \frac12]$ become fixed after applying map $f_1$, points from interval $[\frac14, \frac13]$ become fixed after applying $f_2 \circ f_1$ etc.  Thus $f_{1, \infty}$ contains no distributionally chaotic pairs. \\

For completeness, the opposite implication, DC1 $\implies$ PTE, is not true even in the case of uniformly convergent interval systems. It was shown by Dvořáková in \cite{dvor}. 

\smallskip
\bigskip

\section{Open problem 2}

Paper \cite{bss} deals with generic properties (systems possessing this property form residual set) of non-autonomous systems. Its setting is the same like in the first problem. $(X, d)$ is a compact metric space and the {\it nads} $f_{1,\infty}$ consists of maps $f_n$ from $C_s(X)$, equipped with metric $\rho_{{\rm sup}}$. Moreover, by $\mathcal{F}(X)$ is denoted space of all non-autonomous systems and it is equipped with metric $\rho_{\mathcal{F}}(f_{1, \infty} ,g_{1, \infty}) = \sup_{n}\rho_{{\rm sup}}(f_n,g_n)$. In particular,
\begin{equation}
\rho_{\mathcal{F}}(f_{1, \infty},g_{1, \infty}) = \sup_{n \in \mathbb{N}} \sup_{x \in X} \rho (f_n(x),g_n(x)).
\label{metric}
\end{equation}

There are two main results in \cite{bss}. First one says that infinite topological entropy is generic property for uniformly convergent and equicontinuous non-autonomous systems on interval $I$. The second result states that on the other hand, generic non-autonomous system on a Cantor set has zero topological entropy. Paper \cite{bss} asks following open problem.
\bigskip

{\bf Problem:} On which spaces it hold that DC1 is a generic property? For uniformly convergent and equicontinuous systems, the answer is positive. What about pointwise convergent {\it nads}?\\
\smallskip

Let us denote the set of all convergent {\it nads} on $X$ by $\mathcal{F}_p(X)$. Like in the problem from section \ref{oppr1}, with the supremum metrics $\rho_{\rm sup}$ and $\rho_{\mathcal{F}}$ there is no strictly pointwise convergence. Therefore result from \cite{bss} about genericity of interval {\it nads} with infinite topological entropy is true only for uniformly convergent systems. In $\mathcal{F}_p(I)$, systems with $h(f_{1, \infty})=\infty$ form dense, but not open set. We will see this later, when dealing with distributional chaos. 
\medskip

Genericity of uniformly convergent systems with infinite topological entropy on the interval cannot be generalized to analogous result for DC1 chaotic systems. The difference is caused by the fact that unlike system with $h(f_{1, \infty})>0$ converges uniformly to map $f$ with $h(f)>0$, DC1 chaotic system $f_{1, \infty}$ can converge uniformly to not DC1 limit map (see \cite{dvor}), even on the interval. 
\medskip

It should also be noted, that if we change any of the metrics $\rho_{\rm sup}$ or $\rho_{\mathcal{F}}$ and study strictly pointwise convergent systems, we need to deal with some problems. Space $\mathcal{F}(X)$ equipped with $\rho_{\mathcal{F}}$ is complete (see \cite{bss}), which arises from using uniform metrics on compact spaces. Completeness is important when dealing with generic systems. It is easy to see that if we change any of the two metrics $\rho_{\rm sup}, \rho_{\mathcal{F}}$, space $\mathcal{F}(X)$ is not necessarily complete. In such case we need to specifically demand continuity of all limit maps. Surjectivity of the limit maps is not necessary for the proofs.
\medskip

Results in this section are the following: In Theorem \ref{generic} we show that DC1 {\it nads} are dense in $\mathcal{F}_p(I)$ when we use arbitrary metric. On the other hand, in Lemma \ref{cantor} it is shown that generic {\it nads} on the Cantor set $Q$ is not distributionally chaotic of any type. In Lemma \ref{open} we show that the space of all DC1 chaotic convergent {\it nads} on $I$ is neither open, nor closed.
\medskip

After the auxiliary Remark \ref{aux} we prove our main results.

\begin{remark}
Consider space $C_s(I)$ and let $\varrho$ be a set of all possible metrics on it. Split $\varrho$ into 2 disjoint sets, $\varrho_1$ and $\varrho_2$. Let $E_\varepsilon \subset I$ denote a set with Lebesgue measure less then $\varepsilon$. In $\varrho_1$ we put such metrics $\rho$, for which $\varepsilon \to 0$ implies $\rho(f|_{E_\varepsilon}, g|_{E_\varepsilon} )\to 0$ for all $E_\varepsilon\}$ and for all $f, g \in C_s(I)$. Those metrics never assign positive value to a one-point set and they distinguish between uniform and pointwise convergence (in the meaning that any sequence convergent in $\rho$ can converge either poitwise or uniformly with respect to metric $d$). In $\varrho_1$ there is for example integral metric $\rho_{\rm int}(f,g)=\int_{0}^1 d(f(x), g(x)) dx$.
\medskip

Let $\varrho_2:=\varrho \setminus \varrho_1$. For those metrics there can be sequence of $\{E_\varepsilon\}$ with $\varepsilon \to 0$ such that $\rho(f|_{E_\varepsilon}, g|_{E_\varepsilon} )=c>0$. All sequences of maps converging in $\rho \in \varrho_2$ are uniformly convergent with respect to metric $d$ on $X$. Indeed, for any strictly pointwise convergent sequence $\{f_n\} \to f$ we have
\begin{equation} 
\exists ~\varepsilon>0~~ \forall n ~~~\exists y ~{\rm ~~such~~ that~~} d(f_n(y), f(y))>\varepsilon,
\end{equation}
which means $\rho(f_n,f)>\varepsilon$ for some $\varepsilon>0$, hence $f_n$ is not convergent in $\rho$. The representative member of $\varrho_2$ is metric $\rho_{\rm sup}$. \\
\label{aux}
\end{remark}

\begin{theorem}
Distributionally chaotic systems are dense in $\mathcal{F}_p(I)$ (with continuous limit function), independently of the metric we use.
\label{generic}
\end{theorem}
\smallskip

\begin{proof}
Let first suppose $\mathcal{F}_p(I)$ with metric $\rho_{\mathcal{F}} (f_{1, \infty}, g_{1, \infty})=\sup_{n \in \mathbb{N}} \rho(f_n, g_n)$. For metrics  $\rho \in \varrho_2$, every convergent {\it nads} converges uniformly (with respect to $d$). Since for uniformly convergent interval systems, PTE $\implies$ DC1 (Theorem B in \cite{marta} + Corollary 6.27 in \cite{ruette} + Lemma 3 in \cite{bss}), density of DC1 systems follows by the genericity of systems with infinite topological entropy from \cite{bss}.
\medskip

Therefore in the rest of the proof, we consider metric $\rho \in \varrho_1$. Take {\it nads} $f_{1, \infty}$, where $f_n \to f$ pointwise and $f$ is continuous. To prove the density of DC1 systems, we must find for every $\varepsilon>0$ new system $g_{1, \infty}$ which is distributionally chaotic and such that $\rho(f_n, g_n)<\varepsilon$ for every $n$.
\medskip

First, we find auxiliary system $g'_{1, \infty}$ where $\rho(g'_n, f_n)<\varepsilon/2$ for all $n$, and finally DC1 chaotic system $g_{1,\infty}$, where $\rho(g'_n, g_n)<\varepsilon/2$ for all $n$. It makes $\rho(g_n, f_n)<\varepsilon$. The trick is to choose auxiliary system $g'_{1, \infty}$ such that it converges uniformly. Then we can easily find distributionally chaotic system $g_{1,\infty}$, like in the case of uniform convergence.
\medskip

Set $g'_1 \equiv f_1$ and $g'_n$ for $n>N$ is linear combination between $f_n$ and $f$, such that in all $x$,
\begin{equation}
d(g'_n(x),f(x)) = \frac{1}{2^n} d(f_n(x),f(x)).
\label{d}
\end{equation}
Number $N \in \mathbb{N}$ depends on the system and it will be specified later. In Figure \ref{illustration} there is the illustration of construction of $g'_{1, \infty}$ for one arbitrary system $f_{1, \infty}$ (and $N=1)$.

\begin{figure}[h]
\includegraphics[scale=0.7]{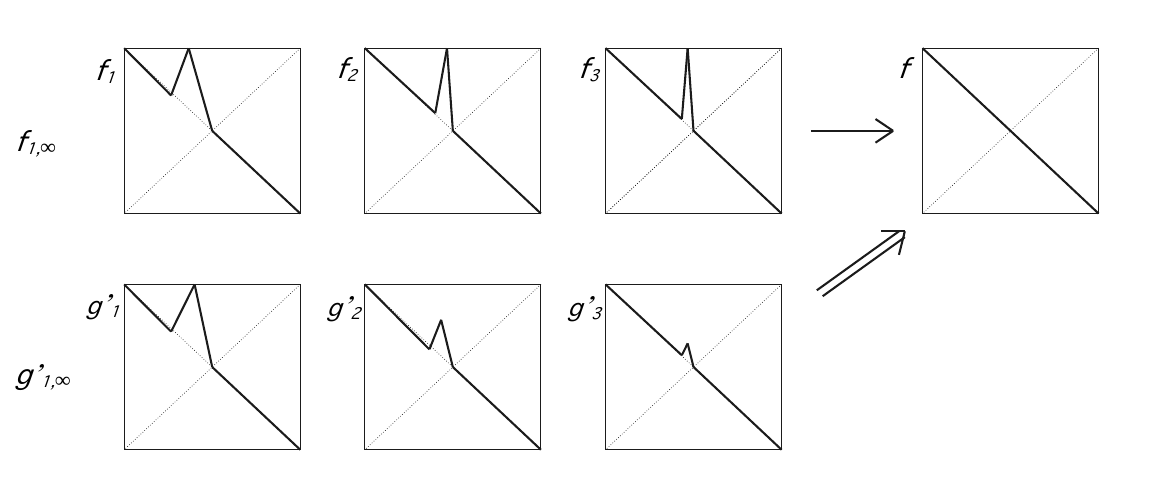}
\caption{Illustration of construction of $g'_{1,\infty}$.}
\label{illustration}
\end{figure}

It is obvious, that
\begin{itemize}
\item all $g'_n$ are continuous and surjective (as linear combinations of continuous and surjective maps),
\item $g'_{1, \infty}$ converges to $f$ ($f_n$ converges to $f$ and $d(g'_n(x), f(x)) \leq d(f_n(x),f(x))$ for all $n$ and $x$),
\item $g'_{1, \infty}$ converges uniformly (all the maps including $f$ are continuous, therefore the distances in (\ref{d}) are shrinking uniformly in all points).
\end{itemize}
\medskip

It remains to show that for our given $\varepsilon$, $\rho(g'_n, f_n)<\frac{\varepsilon}{2}$ for all $n$. First we prove it for $n>N$. System $\{f_n\}$ converges to $f$ pointwise, but by the Egorov's theorem (e.g. \cite{real}), on the large set of points, the convergence is also uniform. In particular, for every $\delta$ there is set $E_\delta$ with $\lambda(E_\delta)<\delta$ such that $f_n$ converges uniformly on $I \setminus E_\delta$. Since by (\ref{d}), $d(g'_n(x), f_n(x))\leq d(f_n(x), f(x))$ for all $x\in I$ and all $n$, we use the same Theorem on system $\{\rho(g'_n, f_n)\}$. Put $\frac{\varepsilon}{2}=\varepsilon'+ \delta$. We get 
\begin{equation}
\forall \delta>0 ~~\exists E_\delta,~ \lambda(E_\delta)<\delta~~ \forall \varepsilon'>0 ~~\exists N~~ \forall n>N,~ d(f_n(x), g'_n(x))<\varepsilon'~~ \forall x \in I \setminus E_\delta
\label{*}
\end{equation}
$$ ({\rm uniform ~~convergence~~ on~~ } I \setminus E_\delta) {\rm ~~and} $$
\begin{equation}
\forall n>N, ~~d(f_n(x), g'_n(x))<1 ~~~~\forall x \in E_\delta {\rm ~~(pointwise~~ convergence~~ on~~ } E_\delta).
\label{**}
\end{equation}
By (\ref{*}) and by the fact that $\rho \in \varrho_1$ we get $\rho(g'_n|_{(I \setminus E_\delta)}, f_n|_{(I \setminus E_\delta)})<\varepsilon' \cdot \lambda(I\setminus E_\delta)<\varepsilon'$. By (\ref{**}), $\rho(g'_n|_{E_\delta}, f_n|_{E_\delta})<1 \cdot \lambda(E_\delta)<\delta$. Together, $\rho(g'_n, f_n)<\varepsilon'+ \delta = \frac{\varepsilon}{2}$ for all $n>N$. Notice that number $N$ was chosen for particular $\varepsilon$, so we can use it as $N$ in (\ref{d}).
\medskip

Finally, let $g'_n \equiv f_n$ for all $n \leq N$. First finite members of $\{g'_n\}$ does not affect its convergence, while $\rho(g'_n, f_n)<\frac{\varepsilon}{2}$. Consequently $\rho_{\mathcal{F}}(f_{1, \infty}, g_{1, \infty})$, which finishes the proof.
\end{proof}
\smallskip

Actually, we think stronger result is true. We think DC1, as well as infinite topological entropy are dense properties not only for convergent non-autonomous systems, but for all {\it nads} on the interval, even for those with not necessarily surjective maps.  We are not able to give a proof, so we present it as an open problem. \\

{\bf Open problem:}  Are DC1 and/or infinite topological entropy systems dense in $\mathcal{F}(I)$ ?\\

We proceed with the next result.

\begin{lemma}
Generic {\it ads} and generic convergent {\it nads} defined on Cantor set $Q$ are not distributionally chaotic.
\label{cantor}
\end{lemma}

\begin{proof}
The proof basically follows from the proofs for zero topological entropy for {\it ads} in \cite{darji} and for {\it nads} in \cite{bss}. In Lemmas 3.2.1 and 3.2.4 due to Darji and D'Aniello in \cite{darji} it is proved that generic map (continuous and surjective) defined on a Cantor set has zero topological entropy. Authors consider map $f$ on the Cantor set and arbitrarily close to it they construct homeomorphism $h$ with zero topological entropy. The constructed homeomorphism behaves similarly to an adding machine on a Cantor set and it is easy to see that there is no distributional chaos. Pairs of points ``travel together" closed in one clopen set, and simultaneously they are constantly apart, trapped in smaller disjoint clopen sets. For more details we refer to \cite{darji}. 
\medskip

Like in \cite{bss} we extend the result to convergent non-autonomous system $f_{1, \infty}$. Let $f_n \to f$. Then $\forall \varepsilon ~~ \exists N ~~ \forall n>N, ~~\rho(f_n,f)<\frac{\varepsilon}{2}$. Using construction from \cite{darji} we find not DC1 map $g$, such that $\rho(g,f)<\frac{\varepsilon}{2}$ and finally let $g_n \equiv f_n$ for $n<N$ and $g_n \equiv g$ otherwise. This  ensures $\sup_n\rho(f_n, g_n)<\varepsilon$ with any metric $\rho$.
\end{proof}

Let $\mathcal{F}_{DC1}(I)$ denote the set of all convergent, DC1 chaotic systems on the interval. The open problem asked, whether to be DC1 is a generic property of {\it nads}. To prove genericity it would be enough (but not necessary) to show that $\mathcal{F}_{DC1}(I)$ is dense and open. In the following Lemma \ref{open} we show that set $\mathcal{F}_{DC1}(I)$, as well as its complement, set of all not DC1 chaotic systems, is neither open, nor closed. Note that Lemma \ref{open} does not say $\mathcal{F}_{DC1}(I)$ is not residual.  This problem stays open. 
\medskip

\begin{lemma}
The set $\mathcal{F}_{DC1}(I)$ is neither open, nor closed in $\mathcal{F}_p(I)$, as well as its complement. 
\label{open}
\end{lemma}

\begin{proof}
A set is open, if every its point has a neighborhood lying in the set. First we show that for $\mathcal{F}_{DC1}(I)$ it is not possible. Let $\rho_{\mathcal{F}}$ now denote arbitrary metric on $\mathcal{F}(I)$. 
\medskip

Suppose $f_{1, \infty} \in \mathcal{F}_{DC1}(I)$ converging to limit map $f$, which is not DC1 chaotic. Such examples exist even for uniformly convergent systems (see \cite{dvor}). We construct, for arbitrary $\varepsilon>0$, system $g_{1, \infty}$ which is not DC1 chaotic and which is $\varepsilon$-close to $f_{1, \infty}$. We use notation $f_{N, \infty}$ to denote system $\{f_n\}_{n=N}^\infty$, starting from $N$ and $\{f\}_{i=1}^\infty$ to denote $\it nds$ consisting of only maps $f$.  Since $f_{1, \infty}$ converges to $f$, we can find for arbitrary $\varepsilon>0$ number $N_0 \in \mathbb{N}$ such that $\rho(f_{N_0, \infty}, \{f\}_{i=N_0}^\infty)<\varepsilon$. We construct $g_{1, \infty}$ as 

\begin{equation}
g_i = \left\{\begin{matrix}
f_i, & 1 \leq i <N_0 & \\ 
f, & N_0\leq i <\infty. & 
\end{matrix}\right.
\end{equation}

Then, 
\begin{equation}
\rho_{\mathcal{F}}(f_{1, \infty}, g_{1,\infty})<\rho_{\mathcal{F}}\left(\{f_i\}_{i=1}^{N_0-1}, \{g_i\}_{i=1}^{N_0-1}\right) + \rho_{\mathcal{F}}\left(\{f_i\}_{i=N_0}^{\infty}, \{g_i\}_{i=N_0}^{\infty}\right)<\varepsilon
\end{equation}

and $g_{1, \infty}$ is not DC1, because its tail is not DC1. Hence $\mathcal{F}_{DC1}(I)$ is not open. 
\medskip

Similarly, we can show, that complement of $\mathcal{F}_{DC1}(I)$, set $\mathcal{F}_p(I) \setminus \mathcal{F}_{DC1}$, is not open either. Consider $f_{1, \infty}$ which is not DC1 chaotic. As such, it must converge to DC1 non-chaotic map $f$. Kawaguchi result in \cite{kawaguchi} states, that the set of DC1 chaotic maps is dense in $C(I)$. Thus, for arbitrary $\varepsilon>0$ we can find DC1 chaotic map $g$, such that $\rho(g, f)<\frac{\varepsilon}{2}$. Similarly to the previous construction we set

\begin{equation}
g_i = \left\{\begin{matrix}
f_i, & 0 \leq i <N_0 & \\ 
g, & N_0\leq i <\infty, & 
\end{matrix}\right.
\end{equation}

where $N_0$ is chosen such that $\rho_{\mathcal{F}}(f_{N_0, \infty}, \{f\}_{i=N_0}^\infty)<\frac{\varepsilon}{2}$. Then $g_{1, \infty}$ is DC1 chaotic, because its tail is DC1 chaotic and we have

\begin{equation}
\rho_{\mathcal{F}}\left(f_{N_0, \infty}, \{g\}_{i=N_0}^\infty\right) \leq \rho_{\mathcal{F}}\left(f_{N_0, \infty}, \{f\}_{i=N_0}^\infty\right) + \rho_{\mathcal{F}}\left(\{f\}_{i=N_0}^{\infty}, \{g\}_{i=N_0}^{\infty}\right)<\frac{\varepsilon}{2} + \frac{\varepsilon}{2}=\varepsilon. \
\end{equation}

Consequently, 

\begin{equation}
\rho_{\mathcal{F}}(f_{1, \infty}, g_{1, \infty}) \leq \rho_{\mathcal{F}}\left(\{f_i\}_{i=1}^{N_0-1}, \{g_i\}_{i=1}^{N_0-1}\right) + \rho_{\mathcal{F}}\left(f_{N_0, \infty}, g_{N_0, \infty}\right)<\varepsilon. 
\end{equation}

Hence, none of the sets $\mathcal{F}_{DC1}(I)$ nor its complement is closed, which finishes the proof.
\end{proof}
\medskip

Similar construction can be used to show that the set of systems with infinite topological entropy is also neither open, nor closed. We can use {\it nads} with infinite topological entropy converging to limit map with zero topological entropy, for example the one from Figure \ref{picture1}. It implies that genericity of infinite topological entropy cannot be easily extended from uniformly convergent systems to $\mathcal{F}_p(I)$. 
\bigskip

To make the situation more clear, we summarize and compare known results on genericity of topologi\-cal entropy and distributional chaos for autonomous and non-autonomous systems on different spaces in the table below. $M$ denotes compact topological manifold, $I$ compact interval and $Q$ the Cantor set. Kawaguchi in \cite{kawaguchi} only proved density of DC1 systems, the openness is known only on $I$. The genericity of DC2(DC3) autonomous systems is proved in \cite{yano} and \cite{down}.\\
\smallskip

\begin{center}
\begin{tabular}{ |c|c| } 
\hline
{\bf autonomous systems} & {\bf non-autonomous systems} \\ 
\hline
\hline

\textbullet ~~ infinite topological entropy is generic in $C(M)$ & \textbullet ~~ infinite topological entropy is generic in $\mathcal{F}_p(I)$\\(Yano, \cite{yano}) & (Balibrea, Smítal, Štefánková, \cite{bss})\\
\textbullet ~~ zero topological entropy is generic in $C(Q)$ & \textbullet ~~ zero topological entropy is generic in $\mathcal{F}_p(Q)$\\ (Darji, D'Aniello, \cite{darji})& (Balibrea, Smítal, Štefánková, \cite{bss})\\
\hline
\textbullet ~~ DC1 (DC2, DC3) systems are dense in $C(M)$ & \textbullet ~~ DC1 (DC2, DC3) systems are dense in $\mathcal{F}_p(I)$\\
(Kawaguchi, \cite{kawaguchi})& (Theorem \ref{generic}) \\
\textbullet ~~ not DC1 systems are generic in $C(Q)$ & \textbullet ~~ not DC1 systems are generic in $\mathcal{F}_p(Q)$\\
(Lemma \ref{cantor}) & (Lemma \ref{cantor})\\
\hline
\end{tabular}
\end{center}
\bigskip

%Extending the result by Kawaguchi to non-autonomous system seems very difficult. He shows that generic (autonomous) system on $M$ has specification property and then, that specification property together with another property implies distributional chaos. For non-autonomous systems not only such implication is not necessarily true, but used properties are not suitably defined in this case. \\

For completeness, in \cite{rouyer} it is proved that if $X$ is a generic compact metric space, then it is perfect and totally disconnected. Therefore it is a Cantor set.\\

\bigskip
\end{document}